\documentclass[12pt]{amsart}

\usepackage{color}
\usepackage{amsfonts}
\usepackage{amscd}
\usepackage{amsmath}
\usepackage{amssymb}
\usepackage{dsfont}
\usepackage{mathabx}

\usepackage{tikz}
\usetikzlibrary{matrix,arrows,decorations.pathmorphing}
\usepackage{pgfkeys}

\newtheorem{Th}{Theorem}[section]

\newtheorem{Prop}[Th]{Proposition}
\newtheorem{Def}[Th]{Definition}
\newtheorem{Cor}[Th]{Corollary}

\newtheorem{Question}{Question}

\newcommand{\dist}{\mathop{\rm dist}}

\newcommand{\ag}{\alpha}

\newcommand{\Lc}{\mathcal{L}}
\newcommand{\Ss}{\mathcal{SS}}
\newcommand{\SC}{\mathcal{SC}}

\newcommand{\cS}{\Phi\mathcal{S}}
\newcommand{\cC}{\Phi\mathcal{C}}

\newcommand{\In}{{\mathcal{I}n}}

\begin{document}

\title[On the perturbation classes problem]{Two classes of operators related to the perturbation classes problem}

\thanks{Supported in part by MINCIN Project PID2019-103961.\newline
}

\author[M. Gonz\'alez]{Manuel Gonz\'alez}
\address{Departamento de Matem\'aticas, Facultad de Ciencias,  Universidad de Cantabria, E-39071 Santander, Espa\~na}
\email{manuel.gonzalez@unican.es}

\author{Margot Salas-Brown}
\address{Escuela de Ciencias Exactas e Ingenier\'\i a, Universidad Sergio Arboleda, Bogot\'a, Colombia}
\email{salasbrown@gmail.com}

\date{\today}

\begin{abstract}
Let $\Ss$ and $\SC$ be the strictly singular and the strictly cosingular operators acting between Banach spaces, and let $P\Phi_+$ and $P\Phi_+$ be the perturbation classes for the upper and the lower semi-Fredholm operators.
We study two classes of operators $\cS$ and $\cC$ that satisfy $\Ss\subset \cS \subset P\Phi_+$ and $\SC\subset \cC\subset P\Phi_-$. We give some conditions under which these inclusions become equalities, from which we derive some positive solutions to the perturbation classes problem for semi-Fredholm operators.
\end{abstract}

\maketitle
\thispagestyle{empty}

\section{Introduction}

The perturbation classes problem asks whether the perturbation classes for the upper semi-Fredholm operators $P\Phi_+$ and the lower semi-Fredholm operators $P\Phi_-$ coincide  with the classes of strictly singular operators $\Ss$ and strictly cosingular operators $\SC$, respectively. This problem was raised in \cite{GohbergMF:60} (see also \cite{CaradusPY:74,Pietsch:80}), and it has a positive answer in some cases \cite{GM-AS:10,GPSB14, GPSB20, Gonzalez:10,Weis:81}, but the general answer is negative in both cases
\cite{Gonzalez:03}, \cite[Theorem 4.5]{GimenezGM-A:12}. However, it remains interesting to find positive answers in special cases because the definitions of $\Ss$ and $\SC$ are intrinsic: to check that $K$ is in one of them only involves the action of $K$, while to check that $K$ is in $P\Phi_+$ or $P\Phi_+$ we have to study the properties of $T+K$ for $T$ in a large set of operators.

In this paper we consider two classes $\cS$ and $\cC$ introduced in \cite{AienaGM-A:12} that satisfy
$$
\Ss(X,Y)\subset \cS(X,Y)\subset P\Phi_+(X,Y) \textrm{ and } \SC(X,Y)\subset \cC(X,Y)\subset P\Phi_-(X,Y),
$$
we study conditions on the Banach spaces $X,Y$ so that some of these four inclusions become equalities, and we derive new positive answers to the perturbation classes problem for semi-Fredholm operators. In the case $\Phi_+(X,Y)\neq \emptyset$, we show that $\cS(X,Y)= P\Phi_+(X,Y)$ when $Y\times Y$ is isomorphic to $Y$, $\Ss(X,Y)= \cS(X,Y)$ when some quotients of $X$ embed in $Y$ (Theorem \ref{equalities}), and adding up we get conditions implying that $\Ss(X,Y)= P\Phi_+(X,Y)$ (Theorem \ref{adding+}). In the case $\Phi_-(X,Y)\neq \emptyset$ we prove similar results (Theorems \ref{equalities-} and \ref{adding-}). We also state some questions concerning the classes $\cS$ and $\cC$. 
\medskip

\emph{Notation.}
Along the paper, $X$, $Y$ and $Z$ denote Banach spaces and $\Lc(X,Y)$ is the set of bounded operators from $X$ into $Y$. We write $\Lc(X)$ when $X=Y$. 
%
Given a closed subspace $M$ of $X$, we denote $J_M$ the inclusion of $M$ into $X$, and $Q_M$ the quotient map from $X$ onto $X/M$. An operator $T\in \Lc(X,Y)$ is an \emph{isomorphism} if there exists $c>0$ such that $\|Tx\|\geq c\|x\|$ for every $x\in X$.

The operator $T\in\Lc(X,Y)$ is \emph{strictly singular}, and we write $T\in\Ss$, when there is no infinite dimensional closed  subspace $M$ of $X$ such that the restriction $TJ_M$ is an isomorphism; and $T$ is \emph{strictly cosingular}, and we write $T\in\SC$, when there is no infinite codimensional closed subspace $N$ of $Y$ such that $Q_NT$ is surjective.
Moreover, $T$ is \emph{upper semi-Fredholm,} $T\in\Phi_+$, when the range $R(T)$ is closed and the kernel $N(T)$ is finite dimensional; $T$ is \emph{lower semi-Fredholm,} $T\in\Phi_-$, when $R(T)$ is finite codimensional (hence closed); $T$ is \emph{Fredholm,} $T\in \Phi$, when it is upper and lower semi-Fredholm; and $T$ is \emph{inessential,} $T\in\In$, when $I_X-ST\in\Phi$ for each $S\in \Lc(Y,X)$.

\section{Preliminaries}

The perturbation class $P\mathcal S$ of a class of operators $\mathcal S$ is defined in terms of its components:

\begin{Def}
Let $\mathcal S$ denote one of the classes $\Phi_+$, $\Phi_-$ or $\Phi$.
For spaces $X,Y$ such that $\mathcal S(X,Y)\neq \emptyset$, 
$$
P\mathcal S(X,Y)= \{K\in \Lc(X,Y) : \textrm{ for each } T\in \mathcal S(X,Y), T+K\in\mathcal S\}.
$$
\end{Def}

We could define $P\mathcal S(X,Y)= \Lc(X,Y)$ when  $\mathcal S(X,Y)$ is empty, but this is not useful.
The components of $P\Phi$ coincide with those of the operator ideal of inessential operators $\In$ when they are defined \cite{AG:98}, but given $S\in\Lc(Y,Z)$ and $T\in\Lc(X,Y)$, $S$ or $T$ in $P\Phi_+$ does not imply $ST\in P\Phi_+$, and similarly for $P\Phi_-$ \cite{Gonzalez:03}. However, the following result holds, and it will be useful for us.

\begin{Prop}\label{prod-PF+}
Suppose that $K\in P\Phi_+(X,Y)$, $L\in P\Phi_-(X,Y)$, $S\in \Lc(Y)$ and $T\in \Lc(X)$. Then $SK, KT\in P\Phi_+(X,Y)$ and $SL, LT\in P\Phi_-(X,Y)$.
\end{Prop}
\begin{proof}
Suppose that $K\in P\Phi_+(X,Y)$ and let $S\in \Lc(Y)$ and $U\in\Phi_+(X,Y)$. If $S$ is bijective then $S^{-1}U\in\Phi_+(X,Y)$, hence $U+SK= S(S^{-1}U +K)\in \Phi_+$; thus $SK \in P\Phi_+(X,Y)$. In the general case, $S=S_1+S_2$ with $S_1,S_2$ bijective; thus $SK= S_1K+S_2K \in P\Phi_+(X,Y)$.

The proof of the other three results is similar.
\end{proof}

\section{The perturbation class for $\Phi_+$}
Given two operators $S,T\in\Lc(X,Y)$, we denote by $(S,T)$ the operator from $X$ into $Y\times Y$ defined by $(S,T)x=(Sx,Tx)$, where $Y\times Y$ is endowed with the product norm $\|(y_1,y_2)\|_1= \|y_1\|+ \|y_2\|$.

Inspired by the results of Friedman \cite{Friedman:02}, the authors of \cite{AienaGM-A:12} defined the following class of operators.

\begin{Def}
Suppose that $\Phi_+(X,Y) \neq\emptyset$ and let $K\in\Lc(X,Y)$. We say that $K$ is \emph{$\Phi$-singular,}  and write $K\in \cS$, when for each $S\in \Lc(X,Y)$, $(S,K)\in \Phi_+$ implies $S\in\Phi_+$.
\end{Def}

The definition of $\cS(X,Y)$ is similar to that of $P\Phi_+(X,Y)$, but the former one is easier to handle because the action of $S$ and $K$ is decoupled when we consider  $(S,K)$ instead of $S+K$.

With our notation, \cite[Theorems 3 and 4]{Friedman:02} can be stated as follows:

\begin{Prop}\label{inclusions}\emph{\cite[Proposition 2.2]{AienaGM-A:12}}
Suppose that $\Phi_+(X,Y)\neq\emptyset$. Then
$$
\Ss(X,Y)\subset \cS(X,Y)\subset P\Phi_+(X,Y).
$$
\end{Prop}

Note that $\Ss$ is an operator ideal but $P\Phi_+$ is not; $P\Phi_+(X,Y)$ is a closed subspace of $\Lc(X,Y)$, and $P\Phi_+(X)$ is an ideal of $\Lc(X)$.

\begin{Prop}
$\Phi \mathcal{S}(X,Y)$ is  closed in $\Lc(X,Y)$
\end{Prop}
\begin{proof}
Let $\{T_n\}$ be a sequence in $\Phi \mathcal{S}(X,Y)$ converging to $T\in \Lc(X,Y)$. Suppose that $S\in \Lc(X,Y)$ and $(S,T) \in \Phi_+(X, Y \times Y)$.
Note that the sequence $(S,T_n)$ converge to $(S,T)$, because $\|(S,T_n)- (S,T)\|= \|T_n-T\|$.

Since $ \Phi_+(X, Y \times Y)$ is an open set, there exists a positive integer $N$ such that $(S,T_N)\in \Phi_+(X,Y\times Y)$. Then $T_N\in \Phi \mathcal{S}(X,Y)$ implies $S \in \Phi_+(X,Y)$. Thus $T\in\Phi \mathcal{S}(X,Y)$. 
\end{proof}

We state some basic questions on the class $\cS$. 

\begin{Question}
Suppose that $\Phi_+(X,Y)\neq \emptyset$.
\begin{itemize}
\item[(a)] Is $\cS(X,Y)$ a  subspace of $\Lc(X,Y)$?
\item[(c)] Is $\cS$ an operator ideal?
\item[(c)] Is Proposition \ref{prod-PF+} valid for $\cS$?
\end{itemize}
\end{Question}

Answering a question in \cite{Friedman:02}, an example of an operator $K\in P\Phi_+\setminus \cS$ was given in \cite[Example 2.3]{AienaGM-A:12}, but we do not know if the other inclusion can be strict.

\begin{Question}\label{SS=cS}
Suppose that $\Phi_+(X,Y)\neq \emptyset$.
Is $\Ss(X,Y)=\cS(X,Y)$?
\end{Question}

A negative answer to Question \ref{SS=cS} would provide a new counterexample to the perturbation classes problem for $\Phi_+$. 

Let us see that the inclusions in Proposition \ref{inclusions} become equalities in some cases.

An infinite dimensional Banach space $Y$ is isomorphic to its  square, denoted $Y\times Y\simeq Y$, in many cases: $L_p(\mu)$ and $\ell_p$ ($1\leq p\leq \infty$), $c_0$, and $C[0,1]$. On the other hand, James' space $J$ and some spaces of continuous functions on a compact like $C[0,\omega_1]$ are not isomorphic to their square, where $\omega_1$ is the first uncountable  ordinal. See \cite{Bess-Pel:60} and \cite{Semadeni:60}.

\begin{Th}\label{equalities}
Suppose that the spaces $X$ and $Y$ satisfy $\Phi_+(X,Y)\neq \emptyset$.
\begin{enumerate}
\item If $Y\times Y\simeq Y$ then $\cS(X,Y)= P\Phi_+(X,Y)$.
\item If every infinite dimensional subspace of $X$ has an infinite dimensional subspace $N$ such that $X/N$ embeds in $Y$ then $\Ss(X,Y)=\cS(X,Y)$.
\end{enumerate}
\end{Th}
\begin{proof}
(1) Let $U:Y\times Y\to Y$ be a bijective isomorphism and let $V,W\in\Lc(Y)$ such that $U(y_1,y_2)=Vy_1+ Wy_2$.

If $K\in P\Phi_+(X,Y)$, for each  $S\in\Lc(X,Y)$ such that $(S,K)\in \Phi_+$ we have $U(S,K)=VS+ WK\in \Phi_+$. By Proposition \ref{prod-PF+}, $WK\in P\Phi_+(X,Y)$. Then $VS\in\Phi_+$, hence $S\in\Phi_+$. Thus we conclude that $K\in \cS(X,Y)$.
\medskip

(2) Let $K\in \Lc(X,Y)$, $K\notin \Ss$. By the hypothesis there exists an  infinite dimensional subspace $N$ of $X$ such that $KJ_N$ is an isomorphism, and there is an isomorphism $U:X/N\to Y$. Then $S=UQ_N\in\Lc(X,Y)$ is not upper semi-Fredholm. We will prove that $K\notin\cS$ by showing that $(S,K)\in \Phi_+$.

Recall that $\|Q_Nx\|= \dist(x,N)$. We can choose the isomorphism $U$ so that $\|Sx\|=\|UQ_Nx\|\geq \dist(x,N)$ for each $x\in X$. Moreover, there is a constant $c>0$ such that $\|Kn\|\geq c\|n\|$ for each $n\in N$.

Let $x\in X$ with $\|x\|=1$, and let $0<\ag<1$ such that $c(1-\ag)= 2\|K\|\ag$.

If $\dist(x,N)\geq \ag$ then $\|Sx\|\geq \ag$.
Otherwise there exists $y\in N$ such that $\|x-y\|<\ag$; hence $\|y\|> 1-\ag$. Therefore
$$
\|Kx\|\geq \|Ky\|-\|K(x-y)\| \geq c(1-\ag)-\|K\|\ag=\|K\|\ag.
$$
Then $\|(S,K)x\|_1=\|Sx\|+\|Kx\|\geq \min\{\|K\|\ag, \ag\}$, hence $(S,K)$ is an isomorphism; in particular $(S,K)\in\Phi_+$, as we wanted to show.
\end{proof}

In the known examples in which $\Ss(X,Y)\neq P\Phi_+(X,Y)$ in \cite {Gonzalez:03,GimenezGM-A:12}, the space $Y$ has a complemented subspace which is  hereditarily indecomposable in the sense of  \cite{ArgyrosF:00, GowersM:93, GowersM:97}. So the question arises.

\begin{Question}
Suppose that $X$ and $Y$ satisfy $\Phi_+(X,Y)\neq \emptyset$ and $Y\times Y\simeq Y$.

Is $\Ss(X,Y)=P\Phi_+(X,Y)$.
\end{Question}

A Banach space $X$ is \emph{subprojective} if every closed infinite dimensional subspace of $X$ contains an infinite dimensional subspace complemented in $X$. The spaces $c_0$, $\ell_p$ ($1\leq p<\infty)$ and $L_q(\mu)$ ($2\leq q<\infty)$ are subprojective \cite{Whitley:64}. See \cite{GalegoGP17,OikhbergS15} for further examples.

\begin{Cor}
Suppose that $\Phi_+(X,Y)\neq \emptyset$ and the space $X$ is subprojective. Then $\Ss(X,Y)= \cS(X,Y)$.
\end{Cor}
\begin{proof}
Every closed infinite dimensional subspace of $X$ contains an infinite dimensional subspace $N$  complemented subspace in $X$; thus $X/N$ is isomorphic to the complement of $N$. Since $\Phi_+(X,Y)\neq \emptyset$, the quotient $X/N$ is isomorphic to a subspace of $Y$ and we can apply Theorem \ref{equalities}.
\end{proof}

The next result is a refinement of Theorem \ref{equalities} that is  proved using the previous arguments.

\begin{Th}\label{adding+}
Suppose that $\Phi_+(X,Y)\neq \emptyset$, $Y\times Y$ embeds in $Y$ and every infinite dimensional subspace of $X$ has an infinite dimensional subspace $N$ such that $X/N$ embeds in $Y$. Then $\Ss(X,Y)=P\Phi_+(X,Y)$.
\end{Th}
\begin{proof}
Since $Y\times Y$ embeds in $Y$, there exist isomorphisms $V,W\in\Lc(Y)$ such that $R(V)\cap R(W)=\{0\}$ and $R(V)+R(W)$ is closed. Hence there exists $r>0$ such that $\|y_1+y_2\|\geq r(\|y_1\|+\|y_2\|)$ for $y_1\in R(V)$ and $y_2\in R(W)$, and clearly we can choose $V,W$ so that $r=1$.

Let $K\in\Lc(X,Y)$ with $K\notin \Ss$. Select an infinite dimensional subspace $M$ of $X$ such that $KJ_M$ is an isomorphism, and let $N$ be an infinite dimensional subspace of $M$ such that there exists an isomorphism $U:X/N\to Y$. We can assume that $\|Uz\|\geq z$ for each $z\in X/N$.

The operator $S=VUQ_N \notin\Phi_+$, and proceeding like in the proof of (2) in Theorem \ref{equalities}, we can show that $S+WK\in\Phi_+$. Then $WK \notin P\Phi_+$, hence $K \notin P\Phi_+$, by Proposition \ref{prod-PF+}.
\end{proof}

\begin{Cor}
If $X$ is separable, $Y\times Y$ embeds in $Y$, and $Y$ contains a copy of $C[0,1]$ then $\Ss(X,Y)=\Phi_+(X,Y)$.
\end{Cor}
\begin{proof}
It is well known that the space $C[0,1]$ contains a copy of each separable Banach space.
\end{proof}

The class $\Phi \mathcal{S}$ is injective in the following sense: 

\begin{Prop}  Given an operator $K \in \Lc(X, Y)$ and an (into) isomorphism $L\in \Lc(Y,Y_0)$, if $LK\in  \Phi\mathcal{S}(X,Y_0)$ then $K\in \Phi \mathcal{S}(X,Y)$. 
\end{Prop}
\begin{proof}
Let $K \in \Lc(X, Y)$ and let $L \in \Lc(X,Y_0)$ be an isomorphism into $Y_0$ such that $LK \in \Phi \mathcal{S}(X,Y_0)$. Take  $S\in \Lc(X,Y)$ and suppose that $(S,K) \in \Phi_+(X, Y \times Y)$. Then $(LS,LK)=(L\times L)(S,K)\in \Phi_+(X, Y_0 \times Y_0)$, where $(L\times L)\in \Lc(X\times X,Y_0 \times Y_0)$ is defined by $(L\times L)(x_1,x_2)=(Lx_1,Lx_2)$.

Since $LK \in \Phi \mathcal{S}(X,Y_0)$ we obtain $LS \in \Phi_+(X, Y_0)$. Therefore $S \in \Phi_+(X, Y)$, hence $K \in \Phi \mathcal{S}(X,Y)$.
\end{proof}

\section{The perturbation class for $\Phi_-$}

Given two operators $S,T\in \Lc(X,Y)$, we denote by $[S,T]$ the operator from $X\times X$ into $Y$ defined by $[S,T](x_1,x_2) =Sx_1+Tx_2$, where $X\times X$ is endowed with the maximum norm $\|(x_1,x_2)\|_\infty= \max\{\|y_1\|, \|y_2\|\}$.

\begin{Def}
Suppose that $\Phi_-(X,Y) \neq\emptyset$ and let $K\in\Lc(X,Y)$. We say that $K$ is \emph{$\Phi$-cosingular,} and write $K\in \cC$, when for each $S\in \Lc(X,Y)$, $[S,K]\in \Phi_-$ implies $S\in\Phi_-$.
\end{Def}

Like in the case of $\cS$, the definition of $\cC(X,Y)$ is similar to that of $P\Phi_-(X,Y)$, but the former one is easier to handle because the action of $S$ and $K$ is decoupled when we consider $[S,K]$ instead of $S+K$.

\begin{Prop}\label{inclusionsC}\emph{\cite[Proposition 2.5]{AienaGM-A:12}}
Suppose that $\Phi_-(X,Y)\neq\emptyset$. Then
$$
\SC(X,Y)\subset \cC(X,Y)\subset P\Phi_-(X,Y).
$$
\end{Prop}

Note that $\SC$ is an operator ideal but $P\Phi_-$ is not; $P\Phi_-(X,Y)$ is a closed subspace of $\Lc(X,Y)$, and $P\Phi_-(X)$ is an ideal of $\Lc(X)$.

\begin{Prop}
$\Phi \mathcal{C}(X,Y)$ is  closed in $\Lc(X,Y)$
\end{Prop}
\begin{proof} 
Let $\{T_n\}$ be a sequence in $\Phi \mathcal{C}(X,Y)$ converging to $T\in \Lc(X,Y)$. Suppose that $S\in \Lc(X,Y)$ and $[S,T] \in \Phi_-(X\times X,Y)$.
Note that the sequence $[S,T_n]$ converge to $[S,T]$.

Since $\Phi_-(X\times X,Y)$ is an open set there exists a positive integer $N$ such that $[S,T_N] \in \Phi_-(X \times X, Y)$. Hence  $T_N \in \Phi \mathcal{C}(X,Y)$ implies  $S \in \Phi_-(X,Y)$.
\end{proof}

\begin{Question}
Suppose that $\Phi_-(X,Y)\neq \emptyset$.
\begin{itemize}
\item[(a)] Is $\cC(X,Y)$ a  subspace of $\Lc(X,Y)$?
\item[(c)] Is $\cC$ an operator ideal?
\item[(c)] Is Proposition \ref{prod-PF+} valid for $\cC$?
\end{itemize}
\end{Question}

Answering a question in \cite{Friedman:02}, an example of an operator $K\in P\Phi_-\setminus \cC$ was given in \cite{AienaGM-A:12}, but we do not know if the other inclusion can be strict.

\begin{Question}\label{SS=cC}
Suppose that $\Phi_-(X,Y)\neq \emptyset$. Is $\SC(X,Y)=\cC(X,Y)$?
\end{Question}

A negative answer to Question \ref{SS=cC} would provide a new counterexample to the perturbation classes problem for $\Phi_-$. 

Next we will show that the inclusions in Proposition \ref{inclusionsC} become equalities in some cases.
\begin{Th}\label{equalities-}
Suppose that the spaces $X$ and $Y$ satisfy $\Phi _-(X,Y) \neq \emptyset$.
\begin{enumerate}
\item If $X\times X \simeq X $ then $\Phi \mathcal{C}(X,Y)=P \Phi_-(X, Y)$.
\item If every infinite codimensional closed subspace of $Y$ is contained in an infinite codimensional closed subspace $N$ which is isomorphic to a quotient of $X$ then $\Phi \mathcal{C}(X,Y)= \SC(X, Y)$.
\end{enumerate}
\end{Th}
\begin{proof}
(1) Let $U:X \rightarrow X \times X$ be a bijective isomorphism and let $U_1,U_2 \in \Lc(X)$ such that $U(x)=(U_1x,U_2x)$.
If $K \in P\Phi_-(X, Y)$, for each $S\in \Lc(X,Y)$ such that $[S,K] \in \Phi_-(X \times X, Y)$ we have $[S,K]U=SU_1+KU_2 \in \Phi_-(X, Y \times Y)$. By Proposition \ref{prod-PF+},  $KU_2 \in P\Phi_-(X, Y)$. Then $SU_1 \in \Phi_-(X, Y)$, hence $S \in \Phi_-(X, Y)$. Thus we conclude that $K \in \Phi \mathcal{C}(X,Y)$.
\medskip

(2) Let $K\in \Lc(X,Y)$ $K\notin \SC$. Then there exists an infinite codimensional closed subspace $M$ of $Y$ such that $Q_NK$ is surjective. By the hypothesis, there exist an infinite codimensional closed subspace $N$ such that $M\subset N$ and a surjective operator $V\in\Lc(X,M)$. 

Observe that $S=J_NV\in\Lc(X,Y)$ is not in $\Phi_-$. We prove that $K\notin \Phi \mathcal{C}$ by showing that $[S,K] \in \Phi_-(X\times X, Y)$. 

Indeed, note that $R(S)=N$. Moreover $Q_NK$ surjective implies $R(K)+N=Y$. Since $R([S,K])= R(S)+ R(K)$, $[S,K]$ is surjective.
\end{proof}

In the known examples in which $\SC(X,Y)\neq P\Phi_-(X,Y)$ in \cite {Gonzalez:03,GimenezGM-A:12}, the space $X$ has a complemented subspace which is  hereditarily indecomposable. So the question arises.

\begin{Question}
Suppose that $\Phi_-(X,Y)\neq \emptyset$ and $X\times X\simeq X$. 

Is $\SC(X,Y)=P\Phi_-(X,Y)$.
\end{Question}

A Banach space $X$ is \emph{superprojective} if each of its infinite codimensional closed 
subspaces is contained in some complemented infinite codimensional  subspace. The spaces $c_0$, $\ell_p$ ($1< p<\infty)$ and $L_q(\mu)$ ($1< q \leq 2)$ are superprojective. See \cite{GalegoGP17,GP16} for further examples.

\begin{Cor} Suppose
 that $\Phi _-(X,Y) \neq \emptyset$ and the space $Y$ is superprojective. Then $\Phi \mathcal{C}(X,Y)= \SC(X, Y)$.
\end{Cor}

\begin{proof}
Every closed infinite codimensional subspace of $Y$ is contained in an infinite codimensional complemented subspace $N$. Since $\Phi _-(X,Y) \neq \emptyset$, there exists $T\in \Lc(X,Y)$ with $R(T)\supset N$, and composing with the projection $P$ on $Y$ onto $N$ we get $R(PT)=N$, and we can apply Theorem \ref{equalities-}.
\end{proof}

The following result is a refinement of the previous results in this section. 

\begin{Th}\label{adding-}
Suppose that $\Phi _-(X,Y) \neq \emptyset$, there exists a surjection from $X \times X$ onto $X$, and every closed infinite codimensional subspace of $Y$ is contained in a closed  infinite codimensional subspace $N$ which is isomorphic to a quotient of $X$. Then $P\Phi_-(X,Y)= \SC(X,Y)$.
\end{Th}
\begin{proof}
$K\in \Lc(X,Y)$, $K\notin \SC$. Then there exists an infinite codimensional subspace $N$ of $Y$ such that $Q_NK$ is surjective. By hypothesis, we can assume that there exists a surjective operator $U:X\to N$. Then $S=J_NU\in \Lc(X,Y)$ is not in $\Phi_-$. Moreover $[S,K]$ is surjective: $R([S,K])= R(S)+R(K)= N+R(K)=Y$, hence $[S,K]\in \Phi_-(X\times X,Y)$.

Let $V:X \rightarrow X \times X$ be a surjection and let $V_1,V_2 \in \Lc(X)$ such that $V(x)=(V_1x,V_2x)$ for each $x\in X$.
Then $[S,K]V= SV_1+KV_2\in \Phi_-(X,Y)$. Since $SV_1\notin \Phi_-$ we get $KV_2\notin P\Phi_-$; hence $K\notin P\Phi_-$ by Proposition \ref{prod-PF+}. 
\end{proof}

\begin{Cor}
If $Y$ is separable, there exists a surjection from $X\times X$ onto $X$, and $X$ has aquotient isomorphic to $\ell_1$ then $\SC(X,Y)=\Phi_-(X,Y)$.
\end{Cor}
\begin{proof}
It is well-known that every separable Banach space is isomorphic to a quotient of $\ell_1$.
\end{proof}

The class $\Phi \mathcal{C}(X,Y)$ is surjective in the following sense:

\begin{Prop}
Given $K\in \Lc(X,Y)$ and a surjective operator $Q\in \Lc(Z,X)$, if   $KQ \in\Phi \mathcal{C}(Z,Y)$ then $K \in \Phi \mathcal{C}(X,Y)$. 
\end{Prop}
\begin{proof}
Let $S\in \Lc(X,Y)$ such that $[S,K] \in \Phi_-(X \times X, Y)$, and let $Q:Z \rightarrow X$ be a surjective operator. Then  the operator $Q\times Q\in \Lc(Z \times Z,X\times X)$ defined by by $(Q\times Q)(a,b)=(Qa,Qb)$ is  surjective. Thus $[S,K](Q\times Q)= [SQ,KQ]$ is in $\Phi_-(Z \times Z,Y)$. Since $KQ \in\Phi \mathcal{C}(Z,Y)$  we obtain $SQ \in \Phi_-(Z , Y)$, hence $S \in \Phi_-(X,Y)$, and we conclude $K \in \Phi \mathcal{C}(X,Y)$. 
\end{proof}

The dual space $(X\times X, \|\cdot \|_\infty)^*$ can be identified with $(X^*\times X^*, \|\cdot\|_1)$ in the obvious way. Hence the conjugate operator  $[S,T]^*$ can be identified with $(S^*,T^*)$. Indeed, for $x^*\in X^*$ and $x\in X$ we have
\begin{eqnarray*}
\langle [S,T]^* x^*, x\rangle &=& \langle x^*, [S,T]x\rangle= \langle x^*, Sx+Tx\rangle\\
&=& \langle S^*x^*+T^*x^*, x\rangle= \langle (S^*,T^*)x^*, x\rangle.
\end{eqnarray*}

As a consequence, $[S,T]\in\Phi_-$ if and only if $(S^*,T^*)\in \Phi_+$. Similarly, $(S,T)^*$ can be identified with $[S^*,T^*]$.

The following result describes the behavior of the classes of $\Phi$-singular and $\Phi$-cosingular operators under duality. 

\begin{Prop} Let $K\in\Lc(X,Y)$. 
\begin{enumerate}
\item If $K^* \in \Phi  \mathcal{S}(Y^*, X^*)$ then $K \in \Phi \mathcal{C}(X,Y)$.
\item If $K^* \in \Phi  \mathcal{C}(Y^*, X^*)$ then $K \in \Phi \mathcal{S}(X,Y)$
\end{enumerate}
\end{Prop}

\begin{proof} (1) Let $S\in \Lc(X,Y)$ such that $[S,K] \in \Phi_-(X \times X, Y)$. Then $[S,K]^* \in \Phi_+$. Since $[S,K]^*\equiv (S^*,K^*)$, we have $(S^*,K^*) \in \Phi_+(Y^*, X^* \times X^*)$, and from $K^* \in \Phi \mathcal{S}(Y^*, X^*)$ we obtain $S^* \in \Phi_+(Y^*, X^*)$; therefore $S \in \Phi_-(X, Y)$, and hence $K \in \Phi \mathcal{C}(X,Y)$.
\medskip

The proof of (2) is similar. 
\end{proof}

\noindent \textsc{Statement.}
The authors have no competing interests to declare that are relevant to the content of this article. 



\begin{thebibliography}{222}

\bibitem{AG:93} P. Aiena, M. Gonz\'alez. \emph{Essentially incomparable Banach spaces and Fredholm theory.} Proc. R. Irish Acad. 93A (1993), 49--59.

\bibitem{AG:98} P. Aiena, M. Gonz\'alez. \emph{On inessential and improjective operators.} Studia Math. 131 (1998), 271-287





\bibitem{AienaGM-A:12}
P.\ Aiena, M.\ Gonz\'alez and A.\ Mart\'\i nez-Abej\'on.
\emph{Characterizations of strictly singular and strictly cosingular operators by perturbation classes.}
Glasgow Math.\ J.\ 54 (2011), 87--96.





\bibitem{ArgyrosF:00} S.A. Argyros, V. Felouzis, \emph{Interpolating hereditarily indecomposable Banach spaces.} J. Amer. Math. Soc. 13 (2000), 243--294.


\bibitem{Bess-Pel:60} C. Bessaga, A. Pe\l czy\'{n}ski. \emph{Banach spaces non-isomorphic to their Cartesian squares. I},
Bull. Acad. Polon. Sci. S\'{e}r. Sci. Math. Astr. Phys. 8 (1960), 77--80.


\bibitem{CaradusPY:74}
S.\ Caradus, W.\ Pfaffenberger and B.\ Yood.
\emph{Calkin algebras and algebras of operators in Banach spaces.}
M.\ Dekker Lecture Notes in Pure \& Appl.\ Math.\ New York, 1974.








\bibitem{Friedman:02}
T.~L.\ Friedman.
\emph{Relating strictly singular operators to the condition   $X < Y \bmod (\mathcal{S}, \mathcal{T})$ and resulting perturbations.}
Analysis (Munich) 22 (2002), 347--354.

\bibitem{GalegoGP17} E.M. Galego, M. Gonz\'alez, J. Pello. \emph{On subprojectivity and superprojectivity.} Results Math. 71 (2017), 1191--1205.

\bibitem{GimenezGM-A:12} J.\ Gim\'enez, M.\ Gonz\'alez, A.\ Mart\'\i nez-Abej\'on. \emph{Perturbation of semi-Fredholm operators on products of Banach spaces.} J.\ Operator Theory 68 (2012), 501--514.

\bibitem{GohbergMF:60} I.~C.\ Gohberg, A.~S.\ Markus and I.~A.\ Feldman. \emph{Normally solvable operators and ideals associated with them.} Bul.\ Akad.\ \v Stiince RSS Moldoven 10 (76) (1960), 51--70. Translation: Amer.\ Math.\ Soc.\ Transl.\ (2) 61 (1967), 63--84.


\bibitem{Gonzalez:03} M.\ Gonz\'alez. \emph{The perturbation classes problem in Fredholm theory.} J.\ Funct.\ Anal.\ 200 (2003), 65--70.





\bibitem{GM-AS:10} M.\ Gonz\'alez, A.\ Mart\'\i nez-Abej\'on and M.\ Salas-Brown.
\emph{Perturbation classes for semi-Fredholm operators on subprojective and superprojective spaces.} Ann.\ Acad.\ Sci.\ Fennicae Math.\ 36 (2011), 481--491.

\bibitem{GP16} M. Gonz\'alez, J. Pello. \emph{Superprojective Banach spaces.} J. Math. Anal. Appl. 437 (2016) 1140--1151.

\bibitem{GPSB14} M. Gonz\'alez, J. Pello, M. Salas-Brown. \emph{Perturbation classes of semi-Fredholm operators in Banach lattices.} J. Math. Anal. Appl. 420 (2014) 792--800.

\bibitem{GPSB20} M. Gonz\'alez, J. Pello, M. Salas-Brown. \emph{The perturbation classes problem for subprojective and superprojective Banach spaces.} J. Math. Anal. Appl. 489 (2020)

\bibitem{Gonzalez:10} M.\ Gonz\'alez and M.\ Salas-Brown. \emph{Perturbation classes for semi-Fredholm operators in $L_p(\mu)$-spaces.} J.\ Math.\ Anal.\ Appl.\ 370 (2010), 11-17.

\bibitem{GowersM:93} W.~T.\ Gowers and B.\ Maurey. \emph{The unconditional basic sequence problem.} J.\ Amer.\ Math.\ Soc.\ 6 (1993), 851--874.

\bibitem{GowersM:97} W.~T.\ Gowers and B.\ Maurey. \emph{Banach spaces with small spaces of operators.} Math. Ann. 307 (1997), 543-568.

\bibitem{FHKT} J. Flores, F.L. Hern\'andez, N.J. Kalton, P. Tradacete. \emph{Characterizations of strictly singular operators on Banach lattices.} J. London Math. Soc. 2 (79) (2009), 612–630.




\bibitem{LT-I}
J.\ Lindenstrauss and L.\ Tzafriri.
\emph{Classical Banach spaces I. Sequence spaces.}
Springer, Berlin, 1977.



\bibitem{OikhbergS15} T. Oikhberg, E. Spinu. \emph{Subprojective Banach spaces.} J. Math. Anal. Appl. 424 (2015) 613--635.

\bibitem{Pietsch:80} A.\ Pietsch.
\emph{Operator ideals.} North-Holland, Amsterdam, 1980.



\bibitem{Semadeni:60} Z. Semadeni. \emph{Banach spaces non-isomorphic to their cartesian squares. II}, Bull. Acad. Polon. Sci. S\'{e}r. Sci. Math. Astr. Phys. 8 (1960), 81--84.




\bibitem{Weis:81} L.\ Weis.
\emph{Perturbation classes of semi-Fredholm operators.}
Math.\ Z.\ 178 (1981), 429--442.

\bibitem{Whitley:64} R.~J.\ Whitley.
\emph{Strictly singular operators and their conjugates.} Trans.\ Amer.\ Math.\ Soc.\ 113 (1964), 252--261.

\end{thebibliography}
\end{document}